\documentclass[11pt]{amsart}

\usepackage{amsmath,graphicx,amssymb,amsthm,amstext}
\usepackage{caption}
\usepackage{subcaption}
\usepackage{tikz}
\usepackage[ruled]{algorithm2e}
\usepackage{listings}
\usepackage{hyperref}

\newcommand*{\QED}{\hfill\ensuremath{\blacksquare}}

\newtheorem{theorem}{Theorem}[section]
\newtheorem{lemma}[theorem]{Lemma}
\newtheorem{corollary}[theorem]{Corollary}

\newtheorem{fact}[theorem]{Fact}
\newtheorem{observation}[theorem]{Observation}

\theoremstyle{definition}
\newtheorem*{definition}{Definition}

\theoremstyle{remark}

\title{A  $(5,5)$-coloring of $K_n$ with few colors}
\author{Alex Cameron}
\address[Alex Cameron]{Department of Mathematics, Statistics, and Computer Science, University of Illinois at Chicago}
\email{acamer4@uic.edu}
\author{Emily Heath}
\address[Emily Heath]{Department of Mathematics, University of Illinois at Urbana-Champaign}
\email{eheath3@illinois.edu}
\date{}

\begin{document}
\maketitle

\begin{abstract}
For fixed integers $p$ and $q$, let $f(n,p,q)$ denote the minimum number of colors needed to color all of the edges of the complete graph $K_n$ such that no clique of $p$ vertices spans fewer than $q$ distinct colors. Any edge-coloring with this property is known as a $(p,q)$-coloring. We construct an explicit  $(5,5)$-coloring that shows that $f(n,5,5) \leq n^{1/3 + o(1)}$ as $n \rightarrow \infty$. This improves upon the best known probabilistic upper bound of $O\left(n^{1/2}\right)$ given by Erd\H{o}s and Gy{\'a}rf{\'a}s, and comes close to matching the best known lower bound $\Omega\left(n^{1/3}\right)$.
\end{abstract}

\section{Introduction}

Let $K_n$ denote the complete graph on $n$ vertices. Fix positive integers $p$ and $q$ such that $1 \leq q \leq {p \choose 2}$. A $(p,q)$-coloring of $K_n$ is any edge-coloring such that every copy of $K_p$ contains edges of at least $q$ distinct colors. Let $f(n,p,q)$ denote the minimum number of colors needed to give a $(p,q)$-coloring of $K_n$.

This function generalizes the classical Ramsey problem. The diagonal Ramsey number $r_k(p)$ is the minimum number of vertices $n$ for which any edge-coloring of $K_n$ with at most $k$ colors will contain a monochromatic copy of $K_p$. So $r_k(p) = n$ implies that $f(n-1,p,2) \leq k$ and $f(n,p,2) \geq k+1$. Similarly, $f(n,p,2)=k$ implies that $r_k(p) > n$ and $r_{k-1}(p) \leq n.$ Therefore, determining $f(n,p,2)$ is equivalent to determining $r_k(p)$ which is well known to be very difficult in general.

Paul Erd\H{o}s originally introduced the function $f(n,p,q)$ in 1981 \cite{erdos1981}, but it was not studied systematically until 1997 when Erd\H{o}s and Andr{\'a}s Gy{\'a}rf{\'a}s \cite{erdos1997} looked at the growth rate of $f(n,p,q)$ as $n \rightarrow \infty$ for fixed values of $p$ and $q$. In particular, they determined the threshold values for $q$ as a function of $p$ for which $f(n,p,q)$ becomes linear in $n$, quadratic in $n$, and asymptotically equivalent to ${n \choose 2}$. They also used the Lov{\'a}s Local Lemma to give a general upper bound, \[f(n,p,q) = O\left(n^{\frac{p-2}{1-q+{p \choose 2}}}\right).\]

\subsection{Determining $f(n,p,p-1) = n^{o(1)}$}

One of the main questions left open by Erd\H{o}s and Gy{\'a}rf{\'a}s \cite{erdos1997} was the determination of a threshold value of $q$ in terms of $p$ for which the function $f(n,p,q)$ becomes polynomial in $n$. They point out a simple induction argument which shows that \[f(n,p,p) \geq n^{1/(p-2)} - 1,\] but could not determine if $f(n,p,p-1) = n^{o(1)}$ even when $p=4$, a problem they called ``the most annoying" of all the small cases.

In 1998, Dhruv Mubayi \cite{mubayi1998note} verified that this is indeed the case when $p=4$ by giving an explicit $(4,3)$-coloring of $K_n$ to show $f(n,4,3)\leq e^{O\left(\sqrt{\log n}\right)}$. Dennis Eichhorn and Mubayi \cite{eichhorn2000note} later used a slight variation of this construction to show that $f(n,p,2\left\lceil \log p\right\rceil-2) \leq e^{O \left(\sqrt{\log n}\right)}$ for all $p \geq 5$ as well. In particular, this showed that $f(n,5,4)$ is also subpolynomial.

Recently, this problem was solved in general when Conlon, Fox, Lee, and Sudakov \cite{CFLS}  provided an explicit coloring which showed that \[f(n,p,p-1) \leq 2^{16p(\log n)^{1-1/(p-2)}\log \log n}.\] This construction is a generalization of the original $(4,3)$-coloring given by Mubayi \cite{mubayi1998note}, and we will use a simplified version of it as part of our $(5,5)$-coloring.

\subsection{Determining $f(n,p,p)$}

As previously stated, we know in general that $f(n,p,p) \geq \Omega \left( n^{1/(p-2)} \right)$. However, the local lemma gives the best general upper bound, \[f(n,p,p) \leq O \left( n^{2/(p-1)} \right).\] Only for $p=3,4$ do we know of a better upper bound.

A $(3,3)$-coloring is equivalent to a proper edge coloring, one in which no two incident edges can have the same color. Therefore, it is well known that \[f(n,3,3) = \left\{ \begin{array}{ll}
n & \quad n \text{ is odd}\\
n-1 & \quad n \text{ is even}
\end{array} \right. \]

In 2004, Mubayi \cite{mubayi2004} provided an explicit $(4,4)$-coloring of $K_n$ with only $n^{1/2}e^{O(\sqrt{\log n})}$ colors. This closed the gap for $p=4$ to \[n^{1/2}-1 \leq f(n,4,4) \leq n^{1/2 + o(1)}.\] His construction was the product of two colorings. The first was his earlier $(4,3)$-coloring which used $n^{o(1)}$ colors. The second was an ``algebraic" coloring that assigned to each vertex a vector from a two-dimensional vector space over a finite field, and then colored each edge with an element from the base field, giving $n^{1/2}$ colors. Some complicating factors needed to be addressed by splitting each of these colors a constant number of times so that ultimately the algebraic part of his coloring used only $O(n^{1/2})$ colors.

One such complications was the need to avoid what Mubayi called a ``striped $K_4$," four vertices with three distinct edge colors where each color is a matching. Interestingly, this particular arrangement can actually be avoided with only $2^{O(\sqrt{\log n})}$ colors as we will show in Section~\ref{modCFLS}.

\subsection{Summary of the $(5,5)$-coloring}

Our $(5,5)$-coloring extends Mubayi's idea of combining a small $(p,p-1)$-coloring with an algebraic coloring to obtain the following result.

\begin{theorem}
As $n \rightarrow \infty$, \[f(n,5,5) \leq n^{1/3}2^{O \left(\sqrt{\log n}\log \log n \right)}.\]
\end{theorem}

We begin in Section~\ref{modCFLS} by considering a particular instance of the general $(p,p-1)$-coloring of Conlon, Fox, Lee, and Sudakov \cite{CFLS} which we will refer to as the CFLS coloring. We show that with few colors this construction avoids certain ``bad" configurations. We then modify it slightly so that it also avoids the striped $K_4$ configuration. By forbidding these specific configurations, we are able to show that there are only three possible edge-colorings of $K_5$ (up to isomorphism) with at most four colors that could still occur with this modified CFLS coloring.

In Section~\ref{MIP}, we define an algebraic coloring which we call the Modified Inner Product (MIP) coloring. Under this construction, each vertex is associated with a vector in a three-dimensional space over a finite field. As in Mubayi's construction \cite{mubayi2004}, each edge is colored with a specific element in the base field. Some slight modifications are needed for special cases, but these will only split each color a constant number of times, ultimately giving $O \left( n^{1/3} \right)$ colors used in the MIP construction.

In Section~\ref{elimination}, we will take the product of these two colorings to get a construction that uses $n^{1/3 + o(1)}$ colors. We will show that under this combined edge-coloring, none of the three remaining configurations can occur.

\section{The CFLS coloring}
\label{modCFLS}

We will not define the CFLS \cite{CFLS} coloring in full generality since only a simple case is needed. We borrow part of the notation used in \cite{CFLS}, but change it somewhat for clarity in this particular instance. Let $n = 2^{\beta^2}$ for some positive integer $\beta$. Associate each vertex of $K_n$ with a unique binary string of length $\beta^2$. That is, we may assume that our vertex set is \[V = \{0,1\}^{\beta^2}.\] For any vertex $v \in V$, let $v^{(i)}$ denote the $i$th block of bits of length $\beta$ in $v$ so that \[v = (v^{(1)},\ldots,v^{(\beta)})\] where each $v^{(i)} \in \{0,1\}^{\beta}$.

Between two vertices $x,y \in V$, the CFLS coloring is defined by \[\varphi_1(x,y) = \left( \left( i, \{x^{(i)},y^{(i)}\}\right), i_1, \ldots, i_{\beta}\right)\] where $i$ is the first index for which $x^{(i)} \neq y^{(i)}$, and for each $k=1,\ldots, \beta$, $i_k = 0$ if $x^{(k)} = y^{(k)}$ and otherwise is the first index at which a bit of $x^{(k)}$ differs from the corresponding bit in $y^{(k)}$.

For convenience, when discussing any edge color $\alpha$, we will let $\alpha_0$ denote the first coordinate of the color (of the form $(i,\{x^{(i)},y^{(i)}\})$) and let $\alpha_k$ denote the index of the first bit difference of the $k$th block for $k=1,\ldots,\beta$.
Furthermore, throughout this section, we will say that two vertices $x$ and $y$ {\em agree at} $i$ if $x^{(i)}=y^{(i)}$ and that $x$ and $y$ {\em differ at} $i$ if $x^{(i)}\neq y^{(i)}$.

\subsection{Avoided configurations}

We will show through the following series of lemmas that the CFLS coloring avoids certain specified arrangements of edge colors.

\begin{lemma}
\label{odd}
The CFLS coloring forbids monochromatic odd cycles.
\end{lemma}

\begin{proof}
Suppose there exists a sequence of distinct vertices, $v_1,\ldots,v_k$, for which $k$ is odd and \[\varphi_1(v_1,v_2) = \varphi_1(v_2,v_3) = \cdots = \varphi_1(v_{k-1},v_k) = \varphi_1(v_k,v_1) = \alpha.\] Let $\alpha_0=(i,\{x,y\})$. Without loss of generality we may assume that $v_1^{(i)} = x$ and $v_2^{(i)} = y$. It follows that \[y = v_2^{(i)} = v_4^{(i)} = \cdots = v_{k-1}^{(i)} = v_1^{(i)} = x,\] a contradiction.
\end{proof}

\begin{figure}
     \begin{subfigure}[b]{0.22\textwidth}
          \centering
          \resizebox{\linewidth}{!}{
          \begin{tikzpicture}
		\draw[very thick, black] (1.0, 0.0) -- (6.123233995736766e-17, 1.0);
		\draw[very thick, red] (1.0, 0.0) -- (-1.0, 1.2246467991473532e-16);
		\draw[very thick, red] (1.0, 0.0) -- (-1.8369701987210297e-16, -1.0);
		\filldraw [black] (1.0, 0.0) circle (1pt);
		\node [right] at (1.0, 0.0) {$a$};
		\filldraw [black] (6.123233995736766e-17, 1.0) circle (1pt);
		\node [above] at (6.123233995736766e-17, 1.0) {$b$};
		\draw[very thick, black] (-1.0, 1.2246467991473532e-16) -- (-1.8369701987210297e-16, -1.0);
		\filldraw [black] (-1.0, 1.2246467991473532e-16) circle (1pt);
		\node [left] at (-1.0, 1.2246467991473532e-16) {$c$};
		\filldraw [black] (-1.8369701987210297e-16, -1.0) circle (1pt);
		\node [below] at (-1.8369701987210297e-16, -1.0) {$d$};
	\end{tikzpicture}
					}
          \caption{}
          \label{fig:A}
     \end{subfigure}
     \begin{subfigure}[b]{0.22\textwidth}
          \centering
          \resizebox{\linewidth}{!}{
          \begin{tikzpicture}
\draw[very thick, black] (1.0, 0.0) -- (6.123233995736766e-17, 1.0);
\draw[very thick, black] (1.0, 0.0) -- (-1.0, 1.2246467991473532e-16);
\draw[very thick, green] (1.0, 0.0) -- (-1.8369701987210297e-16, -1.0);
\filldraw [black] (1.0, 0.0) circle (1pt);
\node [right] at (1.0, 0.0) {$a$};
\draw[very thick, red] (6.123233995736766e-17, 1.0) -- (-1.0, 1.2246467991473532e-16);
\draw[very thick, red] (6.123233995736766e-17, 1.0) -- (-1.8369701987210297e-16, -1.0);
\filldraw [black] (6.123233995736766e-17, 1.0) circle (1pt);
\node [above] at (6.123233995736766e-17, 1.0) {$b$};
\draw[very thick, green] (-1.0, 1.2246467991473532e-16) -- (-1.8369701987210297e-16, -1.0);
\filldraw [black] (-1.0, 1.2246467991473532e-16) circle (1pt);
\node [left] at (-1.0, 1.2246467991473532e-16) {$c$};
\filldraw [black] (-1.8369701987210297e-16, -1.0) circle (1pt);
\node [below] at (-1.8369701987210297e-16, -1.0) {$d$};
\end{tikzpicture}
          }
          \caption{}
          \label{fig:B}
     \end{subfigure}
     \begin{subfigure}[b]{0.22\textwidth}
          \centering
          \resizebox{\linewidth}{!}{
          \begin{tikzpicture}
\draw[very thick, black] (1.0, 0.0) -- (0.30901699437494745, 0.95105651629515353);
\draw[very thick, red] (1.0, 0.0) -- (-0.80901699437494734, 0.58778525229247325);
\filldraw [black] (1.0, 0.0) circle (1pt);
\node [right] at (1.0, 0.0) {$a$};
\draw[very thick, black] (0.30901699437494745, 0.95105651629515353) -- (-0.80901699437494734, 0.58778525229247325);
\filldraw [black] (0.30901699437494745, 0.95105651629515353) circle (1pt);
\node [above] at (0.30901699437494745, 0.95105651629515353) {$b$};
\draw[very thick, black] (-0.80901699437494734, 0.58778525229247325) -- (-0.80901699437494745, -0.58778525229247303);
\draw[very thick, red] (-0.80901699437494734, 0.58778525229247325) -- (0.30901699437494723, -0.95105651629515364);
\filldraw [black] (-0.80901699437494734, 0.58778525229247325) circle (1pt);
\node [left] at (-0.80901699437494734, 0.58778525229247325) {$c$};
\draw[very thick, red] (-0.80901699437494745, -0.58778525229247303) -- (0.30901699437494723, -0.95105651629515364);
\filldraw [black] (-0.80901699437494745, -0.58778525229247303) circle (1pt);
\node [left] at (-0.80901699437494745, -0.58778525229247303) {$d$};
\filldraw [black] (0.30901699437494723, -0.95105651629515364) circle (1pt);
\node [below] at (0.30901699437494723, -0.95105651629515364) {$e$};
\end{tikzpicture}
          }
          \caption{}
          \label{fig:C}
     \end{subfigure}
     \begin{subfigure}[b]{0.22\textwidth}
          \centering
          \resizebox{\linewidth}{!}{
          \begin{tikzpicture}
\draw[very thick, black] (1.0, 0.0) -- (0.30901699437494745, 0.95105651629515353);

\draw[very thick, red] (1.0, 0.0) -- (-0.80901699437494745, -0.58778525229247303);
\draw[very thick, black] (1.0, 0.0) -- (0.30901699437494723, -0.95105651629515364);
\draw[very thick, red] (0.30901699437494745, 0.95105651629515353) -- (-0.80901699437494734, 0.58778525229247325);

\draw[very thick, black] (-0.80901699437494734, 0.58778525229247325) -- (0.30901699437494723, -0.95105651629515364);
\draw[very thick, red] (-0.80901699437494745, -0.58778525229247303) -- (0.30901699437494723, -0.95105651629515364);
\filldraw [black] (1.0, 0.0) circle (1pt);
\node [right] at (1.0, 0.0) {$a$};
\filldraw [black] (0.30901699437494745, 0.95105651629515353) circle (1pt);
\node [above] at (0.30901699437494745, 0.95105651629515353) {$b$};
\filldraw [black] (-0.80901699437494734, 0.58778525229247325) circle (1pt);
\node [left] at (-0.80901699437494734, 0.58778525229247325) {$c$};
\filldraw [black] (-0.80901699437494745, -0.58778525229247303) circle (1pt);
\node [left] at (-0.80901699437494745, -0.58778525229247303) {$d$};
\filldraw [black] (0.30901699437494723, -0.95105651629515364) circle (1pt);
\node [below] at (0.30901699437494723, -0.95105651629515364) {$e$};
\end{tikzpicture}
          }
          \caption{}
          \label{fig:D}
     \end{subfigure}
     \caption{Four configurations avoided by the CFLS coloring.}
 \end{figure}
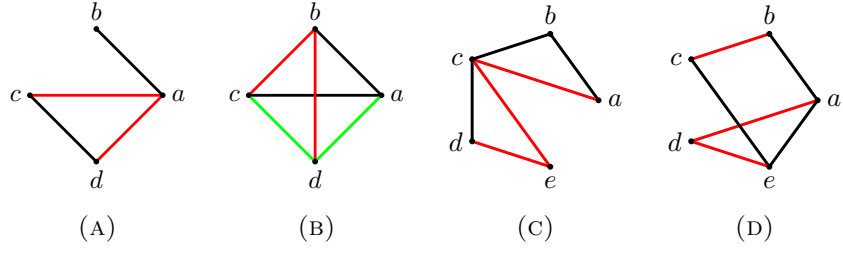

\begin{lemma}
\label{forbidden}
The CFLS coloring forbids four distinct vertices $a,b,c,d \in V$ for which $\varphi_1(a,b)=\varphi_1(c,d)$ and $\varphi_1(a,c)=\varphi_1(a,d)$ (see Figure~\ref{fig:A}).
\end{lemma}

\begin{proof}
Assume towards a contradiction that $\varphi_1(a,b)=\varphi_1(c,d)=\alpha$ and $\varphi_1(a,c)=\varphi_1(a,d)=\gamma$. Let $\alpha_0=(i,\{x,y\})$. Without loss of generality, $a^{(i)}=c^{(i)}=x$ and $b^{(i)}=d^{(i)}=y$. Then $\gamma_i = 0$ since $a$ and $c$ agree at $i$, but $\gamma_i \neq 0$ as $a$ and $d$ differ at $i$, a contradiction.
\end{proof}

\begin{lemma}
The CFLS coloring forbids four distinct vertices $a,b,c,d \in V$ for which $\varphi_1(a,b)=\varphi_1(a,c)$, $\varphi_1(b,d)=\varphi_1(b,c)$, and $\varphi_1(a,d)=\varphi_1(c,d)$ (see Figure~\ref{fig:B}).
\end{lemma}

\begin{proof}
Assume towards a contradiction that we have $\varphi_1(a,b)=\varphi_1(a,c)=\alpha$,  $\varphi_1(b,d)=\varphi_1(b,c)=\gamma$, and $\varphi_1(a,d)=\varphi_1(c,d)=\pi$. Let $\alpha_0=(i,\{x,y\})$, $\gamma_0=(j,\{s,t\})$, and $\pi_0 = (k,\{w,v\})$. Without loss of generality we may assume that $a^{(i)}=x$ and $b^{(i)}=c^{(i)}=y$. Since $b$ and $c$ differ at $j$, then $i \neq j$. Without loss of generality we may assume that $b^{(j)}=s$ and $c^{(j)}=d^{(j)}=t$. So $\pi_j=0$, and hence, $a^{(j)}=t$ since $\varphi_1(a,d)=\pi$. Therefore,  $\alpha_j=0$, which implies that $b^{(j)}=t$, a contradiction since $s \neq t$.
\end{proof}

\begin{lemma}
The CFLS coloring forbids five distinct vertices $a,b,c,d,e \in V$ that contain two monochromatic paths of three edges each that share endpoints: $\varphi_1(a,b)=\varphi_1(b,c)=\varphi_1(c,d)$ and $\varphi_1(a,c)=\varphi_1(c,e)=\varphi_1(e,d)$ (see Figure~\ref{fig:C}).
\end{lemma}

\begin{proof}
Assume towards a contradiction that \[\varphi_1(a,b)=\varphi_1(b,c)=\varphi_1(c,d)=\alpha\] and \[\varphi_1(a,c)=\varphi_1(c,e)=\varphi_1(e,d) = \gamma.\] Let $\alpha_0 = (i,\{x,y\})$ and $\gamma_0=(j,\{s,t\})$. Without loss of generality we may assume that $a^{(i)}=c^{(i)}=x$ and $b^{(i)}=d^{(i)}=y$. Note that $\varphi_1(a,c)=\gamma$ implies $\gamma_i=0$. Then $e^{(i)}=d^{(i)}=y$ and $e^{(i)}=c^{(i)} = x$. So $x=y$, a contradiction.
\end{proof}

\begin{lemma}
The CFLS coloring forbids five distinct vertices $a,b,c,d,e \in V$ for which $\varphi_1(a,b)=\varphi_1(a,e)=\varphi_1(e,c)$ and $\varphi_1(a,d)=\varphi_1(d,e)=\varphi_1(b,c)$ (see Figure~\ref{fig:D}).
\end{lemma}

\begin{proof}
Assume towards a contradiction that $\varphi_1(a,b)=\varphi_1(a,e)=\varphi_1(e,c) = \alpha$ and $\varphi_1(a,d)=\varphi_1(d,e)=\varphi_1(b,c) = \gamma$. Let $\alpha_0=(i,\{x,y\})$. We may assume without a loss of generality that $b^{(i)} = e^{(i)} = x$ and $a^{(i)} = c^{(i)} = y$. We also know that $b^{(k)}=a^{(k)}=e^{(k)}=c^{(k)}$ for all $k<i$. Since $\varphi_1(b,c) = \gamma$, then $\gamma_0 = (i,\{x,y\})$. So either $d^{(i)}=x$ or $d^{(i)}=y$. Therefore, $d$ must agree with either $a$ or $e$ at $i$, a contradiction.
\end{proof}

\subsection{Modified CFLS}

We will now add to the CFLS coloring to avoid the striped $K_4$, an edge-coloring of four distinct vertices $a,b,c,d$ such that every pair of non-incident edges have the same color (see Figure~\ref{stripe}). The CFLS coloring alone will not avoid such arrangements, but the product of $\varphi_1$ with another small edge-coloring, $\varphi_2$, will.

We will define the coloring $\varphi_2$ on the same set of vertices as the CFLS coloring, $V=\{0,1\}^{\beta^2}$. However, we will also need to consider the vertices as an ordered set. Consider each vertex to be an integer represented in binary. Then order the vertices by the standard ordering of the integers. That is, $x<y$ if and only if the first bit at which $x$ and $y$ differ is zero in $x$ and one in $y$. This ordering plays a large role in a recent construction by Mubayi \cite{mubayi2016} for a small case of the hypergraph version of the $(p,q)$-coloring problem. Note that each $\beta$-block is a binary representation of an integer from $0$ to $2^{\beta}-1$,  so these blocks can be considered ordered in the same way. Moreover, note that if $x<y$ and if the first $\beta$-block at which $x$ and $y$ differ is $i$, then it must be the case that $x^{(i)}<y^{(i)}$.

Let $x,y \in V$ such that $x<y$. We define the second coloring as \[\varphi_2(x,y) = \left( \delta_1(x,y),\ldots,\delta_{\beta}(x,y) \right)\] where for each $i$, \[\delta_i(x,y) = \left\{
        \begin{array}{ll}
            -1 & \quad x^{(i)} > y^{(i)} \\
            +1 & \quad x^{(i)} \leq y^{(i)}
        \end{array}
    \right. \]
    
This construction uses $2^{\beta}$ colors. Therefore, the modified CFLS coloring, $\varphi = \varphi_1 \times \varphi_2$, uses \[\beta^{\beta+1}2^{3\beta} = \sqrt{\log n}^{\sqrt{\log n}+1}2^{3\sqrt{\log n}} = 2^{O \left( \sqrt{\log n} \log \log n\right)}\] colors.

\begin{figure}
\centering
\begin{tikzpicture}[scale=1]
\draw[very thick, black] (1.0, 0.0) -- (6.123233995736766e-17, 1.0);
\draw[very thick, red] (1.0, 0.0) -- (-1.0, 1.2246467991473532e-16);
\draw[very thick, green] (1.0, 0.0) -- (-1.8369701987210297e-16, -1.0);
\filldraw [black] (1.0, 0.0) circle (1pt);
\draw[very thick, green] (6.123233995736766e-17, 1.0) -- (-1.0, 1.2246467991473532e-16);
\draw[very thick, red] (6.123233995736766e-17, 1.0) -- (-1.8369701987210297e-16, -1.0);
\filldraw [black] (6.123233995736766e-17, 1.0) circle (1pt);
\draw[very thick, black] (-1.0, 1.2246467991473532e-16) -- (-1.8369701987210297e-16, -1.0);
\filldraw [black] (-1.0, 1.2246467991473532e-16) circle (1pt);
\filldraw [black] (-1.8369701987210297e-16, -1.0) circle (1pt);
\end{tikzpicture}
\caption{A striped $K_4$.}
\label{stripe}
\end{figure}
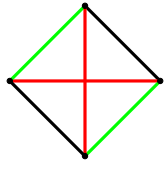
    
\begin{lemma}
The modified CFLS coloring $\varphi$ forbids four distinct vertices $a,b,c,d \in V$ with $\varphi(a,b)=\varphi(c,d)$, $\varphi(a,c)=\varphi(b,d)$, and $\varphi(a,d)=\varphi(b,c)$ (see Figure~\ref{stripe}).
\end{lemma}
    
\begin{proof}
Assume towards a contradiction that a striped $K_4$ can occur. Then, $\varphi_1(a,b)=\varphi_1(c,d) = \alpha$, $\varphi_1(a,c)=\varphi_1(b,d) = \gamma$, and $\varphi_1(a,d)=\varphi_1(b,c)=\pi$. Let $\alpha_0 = (i,\{x,y\})$, $\gamma_0=(j,\{s,t\})$, and $\pi_0=(k,\{v,w\})$. Without loss of generality, assume that $i = \min\{i,j,k\}$. Since exactly one of $d^{(i)}$ and $c^{(i)}$ equals $a^{(i)}$, then either $j=i$ or $k=i$. Moreover, the other index must be strictly greater than $i$. So we may assume that $j=i$ and that $i < k$.

Let $a^{(i)}=d^{(i)}=x$, $b^{(i)}=c^{(i)}=y$, $a^{(k)}=b^{(k)}=v$, and $c^{(k)}=d^{(k)}=w$. Without loss of generality we may assume that $x<y$. This implies that $a,d < b,c$. If $v < w$, then $\delta_k(a,c) = +1$ and $\delta_k(d,b) = -1$. Therefore, $\varphi_2(a,c) \neq \varphi_2(b,d)$, a contradiction. So, it must be the case that $w < v$. But then $\delta_k(a,c)=-1$ and $\delta_k(b,d)=+1$, which yields the same contradiction.
\end{proof}

Note that to eliminate the striped $K_4$ configuration we needed just \[\beta 2^{3 \beta} = \sqrt{\log n} 2^{3\sqrt{\log n}} = 2^{O \left( \sqrt{\log n} \right)}\] colors since only the first coordinate of the CFLS coloring was needed in the proof.

We can now systematically look at all edge-colorings of a $K_5$ up to isomorphism that use no more than four colors and do not contain any of these configurations to get a list of possible ``bad" colorings of a $K_5$ that could survive the modified CFLS coloring. A careful mathematician with a free day could work through these cases by hand. A simple computer program like the inelegant one detailed in Appendix~\ref{script} is easier to verify. However this process is executed, we end up with three possible bad colorings of $K_5$ (see Figure~\ref{badguys}). Avoiding these will require both the CFLS coloring and the MIP coloring defined in Section~\ref{MIP}.

Before we move on from discussing the modified CFLS coloring, we need to point out one nice fact that will be used in Section~\ref{elimination}.

\begin{lemma}
\label{order}
If $a<b<c$, then $\varphi(a,b) \neq \varphi(b,c)$.
\end{lemma}

\begin{proof}
Suppose $\varphi_1(a,b)=\varphi_1(b,c) = \alpha$ and that $\alpha_0 = (i, \{x,y\})$ for $x < y$. Then $a^{(i)} = x$ and $b^{(i)}=y$. But then $c^{(i)}=x$. Therefore, $c < b$, a contradiction.
\end{proof}

\section{The Modified Inner Product coloring}
\label{MIP}

Let $q$ be some odd prime power, and let $\mathbb{F}_q^{\text{*}}$ denote the nonzero elements of the finite field with $q$ elements. The vertices of our graph will be the  three-dimensional vectors over this set, \[V =\left( \mathbb{F}_q^{\text{*}}\right)^3.\] All algebraic operations used in defining the MIP coloring are the standard ones from the underlying field, and $\cdot$ will denote the standard inner product of two vectors, \[x \cdot y = x_1y_1 + x_2y_2 + x_3y_3\] where $x = (x_1,x_2,x_3)$ and $y=(y_1,y_2,y_3)$. Additionally, let $<$ be any linear order on the elements of $\mathbb{F}_q$, and extend this to a linear order on the vectors so that \[x<y \iff x_i < y_i\] where $i \in \{1,2,3\}$ is the first position at which $x_i \neq y_i$.

The MIP coloring will be broken up into two parts, $\chi = \chi_1 \times \chi_2$. The first part $\chi_1$ uses at most $12n^{1/3}$ colors. The second part $\chi_2$ uses only four colors and is used to split up colors from $\chi_1$ in order to avoid one particularly difficult configuration. In this section, we will first define $\chi_1$. Then, after a brief review of the necessary linear algebra concepts, we will prove some key properties of $\chi_1$. Finally, we will define $\chi_2$.

\subsection{The coloring $\chi_1$}

The MIP coloring should be viewed primarily as coloring each edge with the inner product of the two vectors. In Lemma~\ref{collinear} we will show that $\chi_1$ induces a proper edge coloring on every line, not just one-dimensional linear spaces but affine lines as well. This will be one of the key lemmas in showing that our construction avoids the remaining configurations. By itself, the inner product almost accomplishes this goal. However, a problem arises when one vector on a given line is orthogonal to the direction of the line. In this case, that particular vector has the same inner product with all other vectors on the line, so we must give these edges new colors. We accomplish this by replacing the inner product with another function.

The first part of $\chi_1$ labels the type of edge-coloring we will have. For two distinct vectors, $x,y \in V$, let $T(x,y)$ be a function defined by \[T(x,y) = \left\{ \begin{array}{ll}
\text{UP}_1 & \quad x \cdot y = x \cdot x \text{ and } x_1 < y_1\\
\text{UP}_2 & \quad x \cdot y = x \cdot x, x_1 = y_1, \text{ and } x<y\\
\text{DOWN}_1 & \quad x \cdot y \neq x \cdot x, x \cdot y = y \cdot y, \text{ and } x_1 < y_1\\
\text{DOWN}_2 & \quad x \cdot y \neq x \cdot x, x \cdot y = y \cdot y, x_1 = y_1, \text{ and } x < y\\
\text{ZERO} & \quad x \cdot y \not \in \{ x \cdot x, y \cdot y\} \text{ and } x \cdot y = 0\\
\text{DOT} & \quad \text{otherwise}
\end{array} \right. \] Here, the categories $\text{UP}_i$ and $\text{DOWN}_i$ let us know that at least one of the two vectors is orthogonal to the direction of the line between the two, and therefore this edge will need to receive something other than the inner product in the next part of the color. The words UP and DOWN describe the edge from the perspective of the ``special" vertex. For instance, if $x$ is orthogonal to the direction of the line it makes with $y$ and $x<y$, then $x$ looks up the edge to $y$. The need for different categories when $x_1=y_1$ is a technical point. The category DOT stands for the inner product (or the ``dot" product), and ZERO is the special case where the inner product is zero. The need to split the colors with zero inner product is also a technical point.

Let $f_T(x,y): \mathbb{F}_q^3 \rightarrow \mathbb{F}_q$ be a function defined by \[f_T(x,y) = \left\{ \begin{array}{ll}
x_1 + y_1 & \quad T \in \{\text{UP}_1,\text{DOWN}_1,\text{ZERO}\}\\
x_2 + y_2 & \quad T \in \{\text{UP}_2,\text{DOWN}_2\}\\
x \cdot y & \quad T = \text{DOT}
\end{array} \right.\] One final technical point is to differentiate colors based on whether the two vectors are linearly dependent or independent. Let \[ \delta(x,y) = \left\{ \begin{array}{ll}
0 & \quad \{x,y\} \text{ is linearly dependent}\\
1 & \quad \{x,y\} \text{ is linearly independent}
\end{array} \right.\] This is enough to define the coloring. For vertices $x<y$, let $T = T(x,y)$, and set \[\chi_1(x,y) = \left(T,f_T(x,y),\delta(x,y)\right).\]

\subsection{Algebraic definitions and facts}
\label{definitions}

We assume that the reader has some familiarity with basic linear algebra notions such as dimension, linear independence, linear combination, and span. The following definitions and facts are perhaps less familiar. All are  reproduced from definitions and propositions in Chapter 2 of the great \emph{Linear Algebra Methods in Combinatorics} book by L\'{a}szl\'{o} Babai and P\'{e}ter Frankl \cite{linear}.

\begin{definition}
Let $\mathbb{F}^n$ be a vector space, and let $S \subseteq \mathbb{F}^n$ be a set of vectors. The \emph{rank} of $S$ is the dimension of the linear space spanned by $S$.
\end{definition}

\begin{fact}
Let $\mathbb{F}$ be a field, and let $A$ be a $k \times n$ matrix over $\mathbb{F}$. Then the rank of the set of column vectors as vectors in $\mathbb{F}^k$ is equal to the rank of the set of row vectors as vectors in $\mathbb{F}^n$. We know this value as the rank of the matrix $A$, $rk(A)$. \QED
\end{fact}

\begin{definition}
Let $\mathbb{F}^n$ be a vector space. An \emph{affine combination} of vectors $v_1,\ldots,v_k \in \mathbb{F}^n$ is a linear combination $\lambda_1v_1 + \cdots + \lambda_kv_k$ for $\lambda_1,\ldots,\lambda_k \in \mathbb{F}$ such that $\lambda_1 + \cdots + \lambda_k = 1$. An \emph{affine subspace} is a subset of vectors that is closed under affine combinations.
\end{definition}

\begin{fact}
Any affine subspace $U$ is either empty or the translation of some linear subspace $V$. That is, each vector $u \in U$ can be written in the form $u = v+t$ where $v$ is some vector in $V$ and $t$ is a fixed translation vector. \QED
\end{fact}

\begin{definition}
The \emph{dimension} $\text{dim}(U)$ of an affine subspace $U$ is the dimension of the unique linear subspace of which $U$ is a translate. 
\end{definition}

\begin{definition}
Let $\mathbb{F}^n$ be a vector space. Let $v_1,\ldots,v_k \in \mathbb{F}^n$. We say that these vectors are \emph{affine independent} if \[\lambda_1v_1 + \cdots + \lambda_kv_k = 0\] implies that \[\lambda_1 =  \cdots = \lambda_k = 0\] for any $\lambda_1,\ldots,\lambda_k \in \mathbb{F}$ for which $\lambda_1 + \cdots + \lambda_k = 0$. Otherwise, these vectors are \emph{affine dependent}. We say that a set of vectors $S$ is a \emph{basis} for an affine subspace if they are affine independent and every vector in the subspace is an affine combination of vectors in $S$.
\end{definition}

\begin{fact}
A basis of an affine subspace $U$ contains exactly $\text{dim}(U)+1$ elements. \QED
\end{fact}

\begin{fact}
Let $\mathbb{F}^n$ be some vector space. Let $A$ be a $k \times n$ matrix over $\mathbb{F}$ and $b \in \mathbb{F}^n$. Then the solution set to $Ax = b$ is an affine subspace of dimension $n - rk(A)$. \QED
\end{fact}

\begin{definition}
A vector $x \in \mathbb{F}^n$ is \emph{isotropic} if $x \cdot x = 0$. A linear subspace $U \subseteq \mathbb{F}^n$ is \emph{totally isotropic} if $x,y \in U$ implies that $x \cdot y = 0$.
\end{definition}

\begin{fact}
For any nonzero vector $x \in \mathbb{F}^n$, the set of vectors $\{y : x \cdot y = 0\}$ is a linear subspace of $\mathbb{F}^n$ with dimension $n-1$. \QED
\end{fact}

\begin{fact}
\label{funfact}
In a nonsingular inner product space of dimension $n$, every totally isotropic space subspace has dimension $\leq \left\lfloor \frac{n}{2} \right\rfloor$. \QED
\end{fact}

\subsection{Properties of $\chi_1$}

\begin{lemma}
\label{collinear}
The coloring $\chi_1$ induces a proper edge coloring on every one-dimensional affine subspace.
\end{lemma}

\begin{proof}
Let $a,b,c \in \mathbb{F}_q^3$ be three distinct vectors in a one-dimensional affine subspace. Then there exists some $\lambda \in \mathbb{F}_q$ such that $c = \lambda a + (1-\lambda)b$. Suppose towards a contradiction that $\chi_1(a,b) = \chi_1(a,c)$, and let $T=T(a,b)=T(a,c)$. If $T \in \{\text{ZERO},\text{DOT}\}$, then \[a \cdot b = a \cdot (\lambda a + (1-\lambda)b).\] So $\lambda a \cdot (a - b) = 0$. Since $c \neq b$, then $\lambda \neq 0$. Therefore, $a \cdot (a-b) = 0$. But this contradicts the assumption that $T \in \{\text{ZERO},\text{DOT}\}$.

If $T \in \{\text{UP}_1,\text{DOWN}_1\}$, then $f_T(a,b)=f_T(a,c)$ gives \[a_1 + b_1 = a_1 + \lambda a_1 + (1 - \lambda)b_1.\] So $b_1 = a_1$, a contradiction since $T \in \{\text{UP}_1,\text{DOWN}_1\}$ implies that $a_1 \neq b_1$. Similarly, if $T \in \{\text{UP}_2,\text{DOWN}_2\}$, then $a_2 = b_2$ by the same argument and $a_1 = b_1$ by definition. But then either $a_3(a_3-b_3)=0$ or $b_3(b_3-a_3)=0$. Both cases imply that $a_3=b_3$. So $a=b$, a contradiction.
\end{proof}

\begin{definition}
Given a vertex $a \in V$ and an edge-color $A$, let \[N_A(a) = \{x : \chi_1(a,x) = A\}\] be the $A$-neighborhood of $a$
\end{definition}

\begin{observation}
\label{mainobs}
Given a vector $a \in V$ and a color $A=(T,\alpha,i)$, the vectors in $N_A(a)$ all belong to the two-dimensional affine subspace defined by $\{x: f_T(a,x) = \alpha\}$. In particular, this plane can be defined as the solution space to either $a \cdot x = \alpha$ when $T = \text{DOT}$, $(1,0,0) \cdot x = \alpha - a_1$ when $T \in \{\text{ZERO},\text{UP}_1,\text{DOWN}_1\}$, and $(0,1,0) \cdot x = \alpha - a_2$ when $T \in \{\text{UP}_2,\text{DOWN}_2\}$.
\end{observation}

In certain cases, we can actually say something a little stronger.

\begin{lemma}
Given a vector $a \in V$, and a color $A=(T,\alpha,i)$, the vectors of $N_A(a)$ all belong to a one-dimensional affine subspace if one of the following three cases hold:
\begin{enumerate}
\item $T \in \{\text{ZERO},\text{UP}_2,\text{DOWN}_2\}$;
\item $T = \text{UP}_1$ and $a_1 < \alpha - a_1$;
\item $T = \text{DOWN}_1$ and $a_1 > \alpha - a_1$.
\end{enumerate}
\end{lemma}

\begin{proof}
In the first case, if $T=\text{ZERO}$, then every $x \in N_A(a)$ must satisfy the system of linear equations
\begin{align*}
a \cdot x &= 0\\
(1,0,0) \cdot x &= \alpha-a_1.
\end{align*}
Since $a$ contains no zero components, then the rank of $\{a,(1,0,0)\}$ is two. Therefore, the solution space must be a one-dimensional affine subspace. If $T \in \{\text{UP}_2,\text{DOWN}_2\}$, then every $x \in N_A(a)$ must satisfy the system
\begin{align*}
(1,0,0) \cdot x &= a_1\\
(0,1,0) \cdot x &= \alpha-a_2.
\end{align*}
Since $(1,0,0)$ and $(0,1,0)$ are linearly independent, then, as before, the set of solutions is a one-dimensional affine subspace.

In each of the other two cases, we see that every $x \in N_A(a)$ must satisfy the system
\begin{align*}
a \cdot x &= a \cdot a\\
(1,0,0) \cdot x &= \alpha-a_1.
\end{align*}
As before, the solution space must be a one-dimensional affine subspace.
\end{proof}

Therefore, we immediately get the following corollary by Lemma~\ref{collinear}.

\begin{corollary}
\label{mainsub}
Let $a,b,c,d \in V$ be four distinct vertices such that \[\chi_1(a,b)=\chi_1(a,c)=\chi_1(a,d) = (T,\alpha,i).\] The set of vertices $\{b,c,d\}$ span three distinct edge colors under $\chi_1$ if any of the following are true:
\begin{enumerate}
\item $T \in \{\text{ZERO},\text{UP}_2,\text{DOWN}_2\}$;
\item $T = \text{UP}_1$ and $a_1<\alpha-a_1$;
\item $T = \text{DOWN}_1$ and $a_1 > \alpha - a_1$.
\end{enumerate}
\end{corollary}

\begin{figure}
\centering
 \begin{tikzpicture}[scale=1]

\draw[very thick, black] (1.0, 0.0) -- (-0.80901699437494734, 0.58778525229247325);
\draw[very thick, black] (1.0, 0.0) -- (-0.80901699437494745, -0.58778525229247303);
\draw[very thick, black] (1.0, 0.0) -- (0.30901699437494723, -0.95105651629515364);
\draw[very thick, red] (0.30901699437494745, 0.95105651629515353) -- (-0.80901699437494734, 0.58778525229247325);
\draw[very thick, red] (0.30901699437494745, 0.95105651629515353) -- (-0.80901699437494745, -0.58778525229247303);
\draw[very thick, red] (0.30901699437494745, 0.95105651629515353) -- (0.30901699437494723, -0.95105651629515364);

\filldraw [black] (1.0, 0.0) circle (1pt);
\node [right] at (1.0, 0.0) {$a$};
\filldraw [black] (0.30901699437494745, 0.95105651629515353) circle (1pt);
\node [above] at (0.30901699437494745, 0.95105651629515353) {$b$};
\filldraw [black] (-0.80901699437494734, 0.58778525229247325) circle (1pt);
\node [left] at (-0.80901699437494734, 0.58778525229247325) {$c$};
\filldraw [black] (-0.80901699437494745, -0.58778525229247303) circle (1pt);
\node [left] at (-0.80901699437494745, -0.58778525229247303) {$d$};
\filldraw [black] (0.30901699437494723, -0.95105651629515364) circle (1pt);
\node [below] at (0.30901699437494723, -0.95105651629515364) {$e$};
\end{tikzpicture}
\caption{}
\label{gen}
\end{figure}

\begin{lemma}
\label{main}
Let $a,b,c,d,e \in V$ be vectors such that $\{a,b\}$ is linearly independent, $\chi_1(a,c)=\chi_1(a,d)=\chi_1(a,e)$, and $\chi_1(b,c)=\chi_1(b,d)=\chi_1(b,e)$ (see Figure~\ref{gen}). Then the set $\{c,d,e\}$ spans three distinct edge colors.
\end{lemma}

\begin{proof}
Let $\chi_1(a,c)=\chi_1(a,d)=\chi_1(a,e) = A$ and $\chi_1(b,c)=\chi_1(b,d)=\chi_1(b,e) = B$. The result is immediate if either pair $(a,A)$ or $(b,B)$ satisfies the conditions listed in Corollary~\ref{mainsub}. So assume not. If $A = (T_a,\alpha,i)$, then by Observation~\ref{mainobs} we know that $c$, $d$, and $e$ must either satisfy $a \cdot x = \alpha$ or $(1,0,0) \cdot x = \alpha-a_1$. Similarly, if $B = (T_b,\beta,j)$, then $c$, $d$, and $e$ must either satisfy $b \cdot x = \beta$ or $(1,0,0) \cdot x = \beta-b_1$.

Since the sets $\{a,b\}$, $\{a,(1,0,0)\}$, and $\{(1,0,0),b\}$ are all linearly independent, then every case gives us the result immediately except when $T_a,T_b \in \{\text{UP}_1,\text{DOWN}_1\}$. Since we assume that none of the cases from Corollary~\ref{mainsub} hold, then this can only happen when $x \cdot (x - a) = x \cdot (x-b) = 0$ for $x=c,d,e$. In this case, $c$, $d$, and $e$ all satisfy the two linear equations,
\begin{align*}
(a-b) \cdot x &= 0\\
(1,0,0) \cdot x &= \alpha-a_1.
\end{align*}
Hence, $c$, $d$, and $e$ are affine independent, and the result follows from Lemma~\ref{collinear} unless \[a_2-b_2 = a_3 - b_3 = 0.\] But if this is true, then $c \cdot(c-a) = c \cdot (c-b)$ implies that $a=b$, a contradiction.
\end{proof}

\subsection{The coloring $\chi_2$}

Let $U \subseteq \mathbb{F}_q^3$ be a two-dimensional linear subspace. Let $G_U$ be an auxiliary graph where $V(G_U)$ is the set of non-isotropic vectors in $U$, and \[xy \in E(G_U) \iff x \cdot y = 0.\] We wish to show that $G_U$ is bipartite. Note that $x \cdot y = 0$ implies that $\alpha x \cdot \beta y = 0$ for any $\alpha,\beta \in \mathbb{F}_q$. Suppose that $x \cdot z = 0$ for some $z \in V(G_U)$ such that $z \neq \beta y$ for any $\beta \in \mathbb{F}_q$. Then the intersection between  $U$ and the two-dimensional linear subspace orthogonal to $x$ must also be a two-dimensional linear subspace. Therefore, $x$ is contained in its own orthogonal linear subspace. So $x$ is isotropic, a contradiction. Hence, $G_U$ is comprised of disjoint complete bipartite graphs and so is itself bipartite.

For each two-dimensional linear subspace $U$, we label the vertices of $G_U$ with $A_U$ and $B_U$ depending on their part in the bipartition, and then label all isotropic vectors in $U$ with $A_U$ as well.

For any two-dimensional linear subspace $U \subseteq \mathbb{F}_q^3$ and any $x \in U$ we define \[S(x,U) = \left\{ \begin{array}{ll}
A & \quad x \in A_U\\
B & \quad x \in B_U\\
\end{array} \right. \]

For a given vector $a \in V$, and a given color type $T$, define \[a_T = \left\{ \begin{array}{ll}
a & \quad T=\text{DOT}\\
(1,0,0) & \quad T \in \{\text{UP}_1,\text{DOWN}_1,\text{ZERO}\}\\
(0,1,0) & \quad T \in \{\text{UP}_2,\text{DOWN}_2\}
\end{array} \right. \] and let \[U_{a,T} = \{x : a_T \cdot x = 0\}.\] For convenience, let \[a_b = \left\{ \begin{array}{ll}
0 & \quad a_T \cdot a_T = 0\\
(a_T \cdot b)(a_T \cdot a_T)^{-1}a_T & \quad a_T \cdot a_T \neq 0
\end{array} \right. \] for any vectors $a$ and $b$ where $T=T(a,b)$.

Now we can define the second part of the MIP coloring. For any two vectors, $a<b$ with $T=T(a,b)$, let \[\chi_2(a,b) = \left(S(a - b_a,U_{b,T}), S(b - a_b,U_{a,T}) \right).\]

\section{Combining the colorings}
\label{elimination}

Let $n=(q-1)^3$ where $q$ is an odd prime power. To each $\alpha \in \mathbb{F}_q$ we associate the unique element $\alpha' \in \{0,1\}^{\left\lceil \log q \right\rceil}$ which represents in binary the rank of $\alpha$ under the linear order given to the elements of $\mathbb{F}_q$ in Section~\ref{MIP}. Let $\beta$ be the minimum positive integer for which \[3\left\lceil \log q \right\rceil \leq \beta^2. \] We associate each of the $n$ vertices of $K_n$ with a unique vector in $\left(\mathbb{F}^{\text{*}}_q \right)^3$ as in Section~\ref{MIP}. To each vertex $(x_1,x_2,x_3)$, we  associate $(x_1',x_2',x_3',0) \in \{0,1\}^{\beta^2}$ as well, where for each $i$, $x_i'$ is the binary representation of the rank of $x_i$, and $0$ denotes a string of $\beta^2 -3 \left\lceil \log q \right\rceil$ zeros.

Our coloring of the edges of $K_n$ is simply the product of the modified CFLS coloring $\varphi$ and the MIP coloring $\chi$, \[C = \varphi \times \chi.\] Since \[\beta = \Theta \left( \sqrt{3\log q} \right) =  \Theta \left( \sqrt{\log n} \right),\]  it follows that the number of colors used in this combined coloring is at most \[48q \beta 2^{3 \beta} = n^{1/3} 2^{O\left(\sqrt{\log n} \log \log n \right)}\] colors. This bound on the number of colors generalizes to all $n$ by the standard density of primes argument \cite{perelli1984}.

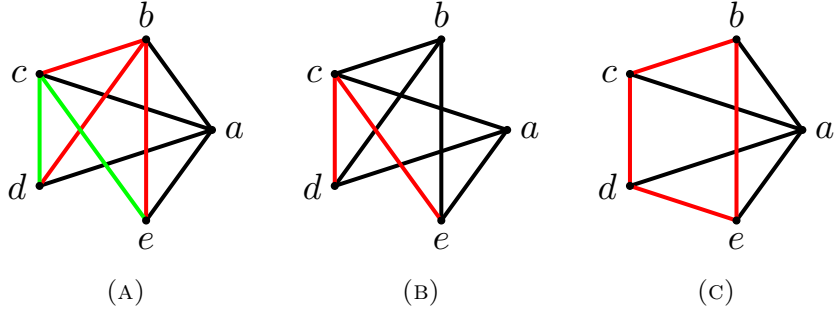
\begin{figure}
     \begin{subfigure}[b]{0.30\textwidth}
          \centering
          \resizebox{\linewidth}{!}{
          \begin{tikzpicture}
\draw[very thick, black] (1.0, 0.0) -- (0.30901699437494745, 0.95105651629515353);
\draw[very thick, black] (1.0, 0.0) -- (-0.80901699437494734, 0.58778525229247325);
\draw[very thick, black] (1.0, 0.0) -- (-0.80901699437494745, -0.58778525229247303);
\draw[very thick, black] (1.0, 0.0) -- (0.30901699437494723, -0.95105651629515364);
\draw[very thick, red] (0.30901699437494745, 0.95105651629515353) -- (-0.80901699437494734, 0.58778525229247325);
\draw[very thick, red] (0.30901699437494745, 0.95105651629515353) -- (-0.80901699437494745, -0.58778525229247303);
\draw[very thick, red] (0.30901699437494745, 0.95105651629515353) -- (0.30901699437494723, -0.95105651629515364);
\draw[very thick, green] (-0.80901699437494734, 0.58778525229247325) -- (-0.80901699437494745, -0.58778525229247303);
\draw[very thick, green] (-0.80901699437494734, 0.58778525229247325) -- (0.30901699437494723, -0.95105651629515364);

\filldraw [black] (1.0, 0.0) circle (1pt);
\node [right] at (1.0, 0.0) {$a$};
\filldraw [black] (0.30901699437494745, 0.95105651629515353) circle (1pt);
\node [above] at (0.30901699437494745, 0.95105651629515353) {$b$};
\filldraw [black] (-0.80901699437494734, 0.58778525229247325) circle (1pt);
\node [left] at (-0.80901699437494734, 0.58778525229247325) {$c$};
\filldraw [black] (-0.80901699437494745, -0.58778525229247303) circle (1pt);
\node [left] at (-0.80901699437494745, -0.58778525229247303) {$d$};
\filldraw [black] (0.30901699437494723, -0.95105651629515364) circle (1pt);
\node [below] at (0.30901699437494723, -0.95105651629515364) {$e$};
\end{tikzpicture}
          }
          \caption{}
          \label{fig:E}
     \end{subfigure}
     \begin{subfigure}[b]{0.30\textwidth}
          \centering
          \resizebox{\linewidth}{!}{
          \begin{tikzpicture}

\draw[very thick, black] (1.0, 0.0) -- (-0.80901699437494734, 0.58778525229247325);
\draw[very thick, black] (1.0, 0.0) -- (-0.80901699437494745, -0.58778525229247303);
\draw[very thick, black] (1.0, 0.0) -- (0.30901699437494723, -0.95105651629515364);
\draw[very thick, black] (0.30901699437494745, 0.95105651629515353) -- (-0.80901699437494734, 0.58778525229247325);
\draw[very thick, black] (0.30901699437494745, 0.95105651629515353) -- (-0.80901699437494745, -0.58778525229247303);
\draw[very thick, black] (0.30901699437494745, 0.95105651629515353) -- (0.30901699437494723, -0.95105651629515364);
\draw[very thick, red] (-0.80901699437494734, 0.58778525229247325) -- (-0.80901699437494745, -0.58778525229247303);
\draw[very thick, red] (-0.80901699437494734, 0.58778525229247325) -- (0.30901699437494723, -0.95105651629515364);

\filldraw [black] (1.0, 0.0) circle (1pt);
\node [right] at (1.0, 0.0) {$a$};
\filldraw [black] (0.30901699437494745, 0.95105651629515353) circle (1pt);
\node [above] at (0.30901699437494745, 0.95105651629515353) {$b$};
\filldraw [black] (-0.80901699437494734, 0.58778525229247325) circle (1pt);
\node [left] at (-0.80901699437494734, 0.58778525229247325) {$c$};
\filldraw [black] (-0.80901699437494745, -0.58778525229247303) circle (1pt);
\node [left] at (-0.80901699437494745, -0.58778525229247303) {$d$};
\filldraw [black] (0.30901699437494723, -0.95105651629515364) circle (1pt);
\node [below] at (0.30901699437494723, -0.95105651629515364) {$e$};
\end{tikzpicture}
          }
          \caption{}
          \label{fig:F}
     \end{subfigure}
     \begin{subfigure}[b]{0.30\textwidth}
          \centering
          \resizebox{\linewidth}{!}{
          \begin{tikzpicture}
\draw[very thick, black] (1.0, 0.0) -- (0.30901699437494745, 0.95105651629515353);
\draw[very thick, black] (1.0, 0.0) -- (-0.80901699437494734, 0.58778525229247325);
\draw[very thick, black] (1.0, 0.0) -- (-0.80901699437494745, -0.58778525229247303);
\draw[very thick, black] (1.0, 0.0) -- (0.30901699437494723, -0.95105651629515364);
\draw[very thick, red] (0.30901699437494745, 0.95105651629515353) -- (-0.80901699437494734, 0.58778525229247325);

\draw[very thick, red] (0.30901699437494745, 0.95105651629515353) -- (0.30901699437494723, -0.95105651629515364);
\draw[very thick, red] (-0.80901699437494734, 0.58778525229247325) -- (-0.80901699437494745, -0.58778525229247303);

\draw[very thick, red] (-0.80901699437494745, -0.58778525229247303) -- (0.30901699437494723, -0.95105651629515364);
\filldraw [black] (1.0, 0.0) circle (1pt);
\node [right] at (1.0, 0.0) {$a$};
\filldraw [black] (0.30901699437494745, 0.95105651629515353) circle (1pt);
\node [above] at (0.30901699437494745, 0.95105651629515353) {$b$};
\filldraw [black] (-0.80901699437494734, 0.58778525229247325) circle (1pt);
\node [left] at (-0.80901699437494734, 0.58778525229247325) {$c$};
\filldraw [black] (-0.80901699437494745, -0.58778525229247303) circle (1pt);
\node [left] at (-0.80901699437494745, -0.58778525229247303) {$d$};
\filldraw [black] (0.30901699437494723, -0.95105651629515364) circle (1pt);
\node [below] at (0.30901699437494723, -0.95105651629515364) {$e$};
\end{tikzpicture}
          }
          \caption{}
          \label{fig:G}
     \end{subfigure}
     \caption{Three configurations not avoided by the modified CFLS coloring.}
\label{badguys}
 \end{figure}

\subsection{The first two configurations}

\begin{lemma}
Any distinct vertices $a,b,c,d,e \in V$ for which $C(a,c) = C(a,d) = C(a,e)$, $C(b,c) = C(b,d) = C(b,e)$, and $C(a,c) \neq C(b,c)$ (see Figure~\ref{gen}) span at least five distinct edge colors.
\end{lemma}

\begin{proof}
Lemma~\ref{odd} implies that neither color between $\{a,b\}$ and $\{c,d,e\}$ can be repeated on the edges spanned by $\{c,d,e\}$. Therefore, if $\{a,b\}$ is linearly independent it follows from Lemma~\ref{main} that $\{a,b,c,d,e\}$ span at least 5 colors.

Otherwise, $b = \lambda a$ for some $\lambda \in \mathbb{F}_q$. If $C(a,b)$ repeats one of the colors from the edges spanned by $\{c,d,e\}$, then this gives us the configuration forbidden by Lemma~\ref{forbidden}. If $C(a,b) = C(a,c)$ or $C(a,b)=C(b,c)$, then all five vectors belong to a one-dimensional linear subspace spanned by $a$ which must be properly edge-colored by Lemma~\ref{collinear}. Therefore, the set of vertices $\{a,b,c,d,e\}$ spans at least 5 colors.
\end{proof}

This immediately shows that the first configuration will not appear under the combined coloring.

\begin{corollary}
Let $a,b,c,d,e \in V$ be five distinct vertices. It cannot be the case that $C(a,b)=C(a,c)=C(a,d)=C(a,e)$, $C(b,c) = C(b,d)=C(b,e)$, and $C(c,d)=C(c,e)$ as in Figure~\ref{fig:E}.
\end{corollary}

The second configuration also will not appear under the combined coloring.

\begin{lemma}
Let $a,b,c,d,e \in V$ be five distinct vertices. It cannot be the case that \[C(a,c)=C(a,d)=C(a,e) = C(b,c) = C(b,d)=C(b,e),\] and $C(c,d)=C(c,e)$ as in Figure~\ref{fig:F}.
\end{lemma}

\begin{proof}
By Lemma~\ref{main}, this can happen only if there exists some $\lambda \in \mathbb{F}_q$ such that $b = \lambda a$. In this case, \[\chi_1(a,c)=\chi_1(a,d)=\chi_1(a,e)=\chi_1(\lambda a,c) = \chi_1(\lambda a,d) = \chi_1(\lambda a,e).\] If this color is in DOT, then $c \cdot a = c \cdot \lambda a$. So either $\lambda=1$, a contradiction, or $c \cdot a = 0$, a contradiction that the color is in DOT. If the color is not in DOT, then it must be the case that $a_1 = \lambda a_1$. Since $a_1 \neq 0$, then this forces $\lambda=1$, a contradiction.
\end{proof}

\subsection{A monochromatic neighborhood that contains a monochromatic $C_4$}

Let $a,b,c,d,e \in V$ be five distinct vertices such that \[C(a,b)=C(a,c)=C(a,d)=C(a,e)=\text{Black},\] and let \[C(b,c)=C(c,d)=C(d,e)=C(e,b)=\text{Red}\] as shown in Figure~\ref{fig:G}. By Lemma~\ref{order} we know that either $b,d < c,e$ or $c,e > b,d$. Similarly, we know that either $a < b,c,d,e$ or $b,c,d,e < a$. So without loss of generality we can say that either $a<b,d<c,e$ or $b,d<c,e<a$. In either case, \[S(b-a_b,U_{a,T}) = S(c-a_c,U_{a,T})\] where $T=T(a,b)=T(a,c)$. 

By Lemma~\ref{mainsub} we know that one of the following three cases must be true:
\begin{enumerate}
\item $\text{Black} \in \text{DOT}$,
\item $\text{Black} \in \text{UP}_1$ such that $b,d<c,e<a$, or
\item $\text{Black} \in \text{DOWN}_1$ such that $a<b,d<c,e$.
\end{enumerate}
This abuses our notation slightly, but the meaning is hopefully clear. For example, $\text{Black} \in \text{DOT}$ means that the first component of the $\chi_1$ part of the color Black is DOT.

\subsubsection{$\text{Black} \in \text{DOT}$}

First, let's assume that $\text{Red} \in \text{DOT}$. Let the inner product part of color Black be $\alpha$ and the inner product part of Red be $\beta$. Note that if either Black or Red encodes linear independence, then $b$, $c$, $d$, and $e$ would all belong to the same one-dimensional linear subspace, a contradiction of Lemma~\ref{collinear}. Also, since $c-e$ satisfies the three linear equations, $a \cdot x =0$, $b \cdot x = 0$, and $d \cdot x = 0$, then $\{a,b,d\}$ cannot be linearly independent since then $c=e$, a contradiction. So there exist nonzero $\lambda_1,\lambda_2 \in \mathbb{F}_q$ such that $d = \lambda_1 a + \lambda_2 b$.

Note that $a \cdot a \neq 0$ since otherwise \[a \cdot d = a \cdot (\lambda_1 a + \lambda_2 b)\] implies that $\lambda_2=1$, and so we can take the inner product of both sides of $d = \lambda_1 a + b$ with $c$ to get that $\lambda_1 \alpha = 0$, a contradiction.

So $a_b = a_c = \alpha(a \cdot a)^{-1} a$, then \[d = \lambda_1' a_b + \lambda_2 b\] where $\lambda_1' = \lambda_1 \alpha^{-1} (a \cdot a).$ Taking the inner product of both side of this with $a$ gives that \[\alpha = (\lambda_1' + \lambda_2)\alpha.\] So it follows that $a_b$, $b$, and $d$ are affine dependent. By the same arguments we can conclude that $a_b$, $c$, and $e$ are also affine dependent.

Note that $(b - d) \cdot (c-e) = 0$. Therefore, $(b - a_b) \cdot (c - a_c)=0$. Since \[S\left(b-a_b, U_{a,\text{DOT}}\right) = S\left(c-a_c, U_{a,\text{DOT}}\right),\] then either $b-a_b$ or $c-a_c$ must be isotropic. If both are isotropic, then  Fact~\ref{funfact} implies that they belong to the same one-dimensional linear space. But if this is true, then $b \cdot b = \alpha^2 (a \cdot a)^{-1}$. Since $(b - a_b) \cdot (c - a_c)=0$ implies that $\beta = \alpha^2 (a \cdot a)^{-1}$, then we can conclude that \[b \cdot (b-c) = 0.\] This contradicts our assumption that $\text{Red} \in \text{DOT}$.

If only one of these vectors is isotropic, then the two-dimensional linear space orthogonal to it must be the same as $U_{a,\text{DOT}}$. Therefore, $a$ itself must be isotropic, which we have already shown is not true.

Now, assume that $\text{Red} \not \in \text{DOT}$, and note that this implies that $b_1 = d_1$ and $c_1=e_1$. If $\text{Red} \in \text{ZERO}$, then \[a \cdot (c-e) = b \cdot (c-e) = d \cdot (c-e) = 0.\] If $a,b,d$ are linearly independent, then $c=e$, a contradiction. If either $b$ or $d$ depends on $a$, then $\delta(a,x) = 0$ for $x = b,c,d,e$ which implies that all five vectors belong to a one-dimensional linear subspace spanned by $a$, contradicting Lemma~\ref{collinear}. If $d = \lambda b$ for some $\lambda \in \mathbb{F}_q$, then $b_1 = d_1 = \lambda b_1$. So either $b_1=0$ or $\lambda = 1$, both contradictions. So we must assume that $d = \lambda_1 a + \lambda_2 b$ for nonzero $\lambda_1,\lambda_2 \in \mathbb{F}_q$. But then
\begin{align*}
d \cdot c &= \lambda_1 (a \cdot c) + \lambda_2 (b \cdot c)\\
0 &= \lambda_1 (a \cdot c)
\end{align*}
Since $\lambda_1 \neq 0$, then $a \cdot c = 0$ which implies that $\text{Black} \not \in \text{DOT}$, a contradiction.

If $\text{Red} \in \text{UP}_2 \cup \text{DOWN}_2$, then $b_1=c_1=d_1=e_1$. So all four vectors $b,c,d,e$ satisfy the linear equations $a \cdot x = \alpha$ and $(1,0,0) \cdot x = b_1$. Therefore, $b$, $c$, $d$, and $e$ all belong to a one-dimensional affine subspace, a contradiction of Lemma~\ref{collinear}.

If $\text{Red} \in \text{UP}_1$, then \[b \cdot (b-c) = b \cdot (b-e) = d \cdot (d-c) = d \cdot (d-e) = 0\] since we assume that $b,d<c,e$. So \[ \left(  \begin{array}{ccc}
a_1 & a_2 & a_3\\
b_1 & b_2 & b_3\\
b_1 & d_2 & d_3 \end{array} \right) \left( \begin{array}{c}
0\\
c_2-e_2\\
c_3 - e_3 \end{array}  \right) = 0.\] Therefore, if any two of $(a_2,a_3)$, $(b_2,b_3)$, and $(d_2,d_3)$ are linearly independent as vectors in $\mathbb{F}_q^2$, then $c=e$, a contradiction. Hence, there must exist $\lambda_1, \lambda_2 \in \mathbb{F}_q$ such that $(a_2, a_3) = \lambda_1 (b_2,b_3)$ and $(d_2,d_3) = \lambda_2 ( b_2, b_3)$.

From the equations of the form $x \cdot (x-c) = 0$ for $x=b,d$ we get
\begin{align*}
b_1(b_1-c_1) +b_2^2 + b_3^2 - b_2c_2 - b_3c_3 &= 0\\
b_1(b_1-c_1) +\lambda_2^2b_2^2 + \lambda_2^2b_3^2 - \lambda_2 b_2c_2 - \lambda_2 b_3c_3 &= 0
\end{align*}
So it follows that \[c_3 = b_3^{-1} \left(b_1(b_1-c_1) + b_2^2 + b_3^2 - b_2c_2\right)\] which in turn gives that \[(1-\lambda_2)b_1(b_1-c_1) = \lambda_2(1-\lambda_2)(b_2^2+b_3^2).\] Since $\lambda_2 \neq 1$, then \[b_1(b_1-c_1) = \lambda_2(b_2^2+b_3^2).\] Since $b_1 \neq 0$ and $b_1 \neq c_1$, then $b_2^2+b_3^2 \neq 0$.

Now, since $\text{Black} \in \text{DOT}$, we get
\begin{align*}
a \cdot b &= a \cdot d\\
a_1b_1 + \lambda_1 b_2^2 + \lambda_1 b_3^2 &= a_1 b_1 + \lambda_1 \lambda_2 b_2^2 + \lambda_1 \lambda_2 b_3^2\\
\lambda_1 (1 - \lambda_2) (b_2^2 + b_3^2) &= 0
\end{align*}
But this is a contradiction, since none of these three terms are zero. If $\text{Red} \in \text{DOWN}_1$, then we swap $b,d$ and $e,c$ in the previous argument to obtain the same contradiction.

\subsubsection{$\text{Black} \not\in \text{DOT}$}

In this case, either $b,d < c,e < a$ and $\text{Black} \in \text{UP}_1$, or $c,e > b,d > a$ and $\text{Black} \in \text{DOWN}_1$. In both cases, \[b \cdot (b-a) = c \cdot (c-a) = d \cdot (d-a) = e \cdot (e-a) = 0.\] Moreover, $b_1=c_1=d_1=e_1$, which implies that \[\text{Red} \in \text{ZERO} \cup \text{UP}_2 \cup \text{DOWN}_2 \cup \text{DOT}.\]

If $\text{Red} \in \text{ZERO}$, then \[b \cdot (c-e) = d \cdot (c-e) = 0.\] Therefore, \[\left( \begin{array}{cc}
b_2 & b_3\\
d_2 & d_3
\end{array} \right) \left( \begin{array}{c}
c_2 - e_2\\
c_3 - e_3
\end{array} \right) = 0.\] So $(d_2,d_3) = \gamma (b_2,b_3)$ for some $\gamma \in \mathbb{F}_q$, and $b \cdot c = d \cdot c$ gives \[b_1^2 +c_2 b_2 +c_3 b_3 = b_1^2 + \gamma c_2b_2 + \gamma c_3b_3.\] Thus, $(1-\gamma)(c_2b_2+c_3b_3) = 0$. Therefore, either $\gamma = 1$, a contradiction since $b \neq d$, or $c_2b_2+c_3b_3=0$, also a contradiction since this implies that $b \cdot c = b_1^2 \neq 0$.

If $\text{Red} \in \text{UP}_2 \cup \text{DOWN}_2$, then we have $b_2=d_2$, $c_2=e_2$, and either $b \cdot (b-c)=b\cdot(b-e)=0$ or $c \cdot (c-b)=c\cdot(c-d)=0$. In the first case, $b \cdot (e-c) = 0$ so $b_3(e_3-c_3) = 0$. So either $b_3=0$ or $e_3=c_3$, both contradictions. Similarly, in the second case, $c \cdot (b-d)=0$ means that $c_3(b_3-d_3)=0$, which gives the same contradictions.

Finally, if $\text{Red} \in \text{DOT}$, then $b \cdot (c-e) = 0$ and $d \cdot (c-e) = 0$. So \[ \left( \begin{array}{cc}
b_2 & b_3\\
d_2 & d_3
\end{array} \right) \left( \begin{array}{c}
c_2-e_2\\
c_3-e_3
\end{array} \right) = 0.\] Therefore, either $c=e$, a contradiction, or $(d_2,d_3) = \lambda_1 (b_2,b_3)$ for some nonzero $\lambda \in \mathbb{F}_q$.

If $\beta$ is the inner product represented by Red, then we get that
\begin{align*}
\beta &= b_1^2 + b_2c_2 + b_3c_3\\
\beta &= b_1^2 + \lambda_1b_2c_2 + \lambda_1b_3c_3\\
\end{align*}
So, \[(1-\lambda_1)(b_2c_2+b_3c_3) = 0.\] Therefore, either $\lambda_1=1$ or $b_2c_2 + b_3c_3=0$. If $\lambda_1=1$, then $b=d$, a contradiction. So we must assume that $b_2c_2 + b_3c_3=0$. But then \[(b - (b_1,0,0)) \cdot (c-(b_1,0,0)) = 0.\] Since we know that \[S(b-(b_1,0,0),U_{a,T})=S(c-(b_1,0,0),U_{a,T}),\] then it must be that either $(0,b_2,b_3)$ or $(0,c_2,c_3)$ is isotropic. If these vectors are linearly independent, then they cannot both be isotropic by Fact~\ref{funfact}.

Moreover, the isotropic vector is orthogonal to the linear subspace $U_{a,T}$. Hence, it must be linearly dependent on $(1,0,0)$, a contradiction. Therefore, $(0,b_2,b_3)$ and $(0,c_2,c_3)$ belong to a totally isotropic one-dimensional linear subspace. So $b_2^2 + b_3^2 = 0$ and $c = (b_1,\lambda_2 b_2, \lambda_2 b_3)$ for some $\lambda_2 \in \mathbb{F}_q$. But then $b \cdot (b-c) = 0$, a contradiction of the assumption that $\text{Red} \in \text{DOT}$.

\section{Conclusion}

This construction provides additional evidence that a general strategy of combining a $(p,p-1)$-coloring with a $\mathbb{F}_q^{p-2} \rightarrow \mathbb{F}_q$ algebraic coloring might show that $f(n,p,p) \leq n^{1/(p-2) + o(1)}$. However, both Mubayi's proof for his $(4,4)$-coloring \cite{mubayi2004} and our proof for the $(5,5)$-coloring require case-checking. Already, in this paper we found it far easier to appeal to an algorithm rather than present a logical elimination of all cases, but this problem will quickly become intractable as $p$ increases. Some general principles will need to be identified before we can demonstrate that the analogous constructions work for all $p$.

Even if this type of construction were to demonstrate such a bound in general, a subpolynomial yet significant gap between the lower and upper bounds persists even for $p=4$. It would be nice to find a way to avoid including the CFLS coloring and tighten  the upper bound, or, perhaps more interestingly, show that the lower bound can be increased.

\section*{Acknowledgements}

This work was supported in part by NSF-DMS Grants 1604458, 1604773, 1604697 and 1603823, ``Collaborative Research: Rocky Mountain - Great Plains Graduate Research Workshops in Combinatorics." Thank you to Sam Cole and Florian Pfender for all of the brainstorming they did with us when we first started thinking about this problem, to Bernard Lidick\'{y} and his computer for narrowing the problem cases down for us early on, to Dhruv Mubayi for introducing one of us to this problem and for chatting about it from time to time, and to everyone involved with the GRWC for creating the time and space for new researchers to collaborate on problems like this one.

\appendix
 
\section{Algorithm for reducing cases}
\label{script}
The following algorithm is not difficult to verify, so we present it here without proof. The specific implementation we rely on is a Python script that can be found (with comments) at \url{http://homepages.math.uic.edu/~acamer4/EdgeColors.py}.

Suppose we want to find every edge-coloring, up to isomorphism, of $K_n$ that uses at most $m$ colors and does not contain a copy of any $F \in \mathcal{F}$, a list of edge-colored complete graphs on $n$ or fewer vertices. The algorithm takes $\mathcal{F}$, $n$, and $m$ as input and returns a list $\mathcal{R}$ of edge-colorings of $K_n$ satisfying these requirements.

For each $k=3,\ldots,n$, the algorithm creates a list $L_k$ of acceptable edge-colorings of $K_k$ by adding a new vertex to each $K_{k-1}$ listed in $L_{k-1}$ (where $L_2$ is the list of exactly one $K_2$ with its single edge given color $1$), and then coloring the $k-1$ new edges in all possible ways from the color set $[m]$. For each graph in $L_{k-1}$ and each way to color the new edges, we test the resulting graph to see if it contains any of the forbidden edge-colorings. If it does, then we move on. If not, then we test it against the new list $L_k$ to see if it is isomorphic to any of the colorings of $K_k$ already on the list. If it is, then we move on. Otherwise, we add it to the list $L_k$. The algorithm terminates when it has tested all colorings of $K_n$.\\

\begin{algorithm}[H]
\SetAlgoLined
\KwData{number of vertices $n$; maximum number of colors $m$; list of forbidden colorings $\mathcal{F}$}
 
initialize $L_2$ as list containing one $K_2$ with its edges colored 1\;
\For{$k=3,\ldots,n$}{
initialize empty list $L_k$\;
\For{$H \in L_{k-1}$}{
\For{each function $f:[k-1] \rightarrow [m]$}{
let $G$ be $K_{k}$ with edge-colors same as $H$ on the first $k-1$ vertices and color $f(i)$ on edge $ki$ for $i=1,\ldots,k-1$\;
\If{$G$ contains no element of $\mathcal{F}$ and is isomorphic to no element of $L_k$}{add $G$ to the list $L_k$}
}
}
}
\Return{$L_n$}
 \caption{List all edge-colorings with no forbidden subcoloring}
\end{algorithm}

\bibliography{ppconstruction}
\bibliographystyle{plain}

\end{document}